\documentclass[11pt]{article}
\usepackage{epsf}

\usepackage{amsthm}
\usepackage{amssymb}
\usepackage{amsmath}
\usepackage{authblk}

\usepackage[top=1.5cm,bottom=1.5cm,left=2.5cm,right=2.5cm,includehead,includefoot
           ]{geometry}

\makeatletter

\numberwithin{equation}{section}
\newtheorem{theorem}{Theorem}[section]
\newtheorem{lemma}[theorem]{Lemma}
\newtheorem{corollary}[theorem]{Corollary}
\newtheorem{remark}[theorem]{Remark}

\newtheorem{proposition}[theorem]{Proposition}

\title{Jordan and Jordan Higher All-derivable Points of Some Algebras}
\author[,a]{Jiankui Li\footnote{Corresponding author. Email address: jiankuili@yahoo.com}}
\author[b]{Zhidong Pan}
\author[a]{Qihua Shen}
\affil[a]{\small Department of Mathematics, East China University of Science and Technology\\

Shanghai 200237, PR China}

\affil[b]{\small Department of Mathematics, Saginaw Valley State University\\

University Center, MI 48710, USA}
\date{}
\begin{document}

\maketitle \abstract
 In this paper, we characterize Jordan derivable mappings in terms
 of Peirce decomposition and determine Jordan all-derivable points
 for some general bimodules. Then we generalize the results to the
 case of Jordan higher derivable mappings. An immediate application
 of our main results shows that for a nest $\mathcal{N}$ on a Banach
 $X$ with the associated nest algebra $alg\mathcal{N}$, if there exists
 a non-trivial element in $\mathcal{N}$ which is complemented in $X$,
  then every $C\in alg\mathcal{N}$ is a Jordan all-derivable point of
  $L(alg\mathcal{N}, B(X))$ and a Jordan higher all-derivable point of $L(alg\mathcal{N})$.
\

{\sl Keywords} : Jordan all-derivable point, Jordan derivation, Jordan
higher all-derivable point, Jordan higher derivation; nest algebra\

{\sl 2010 AMS classification} : Primary 47L35; Secondly 16W25 \

\section{Introduction}\

Let $\mathcal{A}$ be a unital algebra and  $\mathcal{M}$ be a unital
$\mathcal{A}$-bimodule. We denote $C(\mathcal{A},\mathcal{M})=\{M\in
\mathcal{M}: AM=MA~ \textrm{for~ every}~A\in \mathcal{A}\}$ and
$L(\mathcal{A},\mathcal{M})$ the set of all linear mappings from
$\mathcal{A}$ to $\mathcal{M}$. When $\mathcal{M}=\mathcal{A}$, we
relabel $L(\mathcal{A},\mathcal{M})$ as $L(\mathcal{A})$. Let
$\delta\in L(\mathcal{A},\mathcal{M})$. $\delta$ is called a
\textit{derivation} if $\delta(AB)=\delta(A)B+A\delta(B)$ for all
$A,B\in \mathcal{A}$; it is a \textit{Jordan derivation} if
$\delta(AB+BA)=\delta(A)B+A\delta(B)+\delta(B)A+B\delta(A)$ for all
$A,B\in \mathcal{A}$; it is a \textit{generalized derivation}
if there exists an $M_{\delta}\in C(\mathcal{A},\mathcal{M})$ such
that
$\delta(AB)=\delta(A)B+A\delta(B)-M_{\delta}AB$
for all $A,B\in \mathcal{A}$. For any fixed $M\in \mathcal{M}$, each
mapping of the form $\delta_{M}(A)=MA-AM$ for every $A\in
\mathcal{A}$ is called an \textit{inner derivation}. Clearly each
inner derivation is a derivation and each derivation in a Jordan
derivation. But the converse is not true in general. The questions
of characterizing derivations and Jordan derivations have received
considerable attention from several authors, who revealed the
relations among derivations, Jordan derivations as well as inner
derivations (see for example
\cite{CHDM,semiprime,prime,Johnson,LU2,Moore,JDOTA}, and the
references therein).

In general there are two directions in the study of the local
actions of derivations of operator algebras. One is the well known
local derivation problem (see for example
\cite{local3,local2,local1,Zhu6,Zhu5}). The other is to study
conditions under which derivations of operator algebras can be
completely determined by the action on some subsets of operators
(see for example \cite{CHDM,Chebotar, LI1,Hou2,  ZHOU, ZHUJUN2}). A
mapping $\delta\in L(\mathcal{A},\mathcal{M})$ is called a
\textit{Jordan derivable mapping at $C\in \mathcal{A}$} if
$\delta(AB+BA)=\delta(A)B+A\delta(B)+\delta(B)A+B\delta(A)$ for all
$A,B\in \mathcal{A}$ with $AB=C$. It is obvious that a linear
mapping is a Jordan derivation if and only if it is Jordan derivable
at all points. It is natural and interesting to ask the question
whether or not a linear mapping is a Jordan derivation if it is
Jordan derivable only at one given point. If such a point exists, we
call this point a Jordan all-derivable point. To be more precise, an
element $C\in \mathcal{A}$ is called a \textit{Jordan all-derivable
point} of $L(\mathcal{A},\mathcal{M})$ if every Jordan derivable
mapping at $C$ is a Jordan derivation. It is quite surprising that
there do exist Jordan all-derivable points for some algebras. An and
Hou \cite{Hou} show that under some mild conditions on unital prime
ring or triangular ring $\mathcal{A}$, $I$ is a Jordan all-derivable
point of $L(\mathcal{A})$. Jiao and Hou \cite{Hou3} study Jordan
derivable mappings at zero point on nest algebras. Zhao and Zhu
\cite{ZHUJUN} prove that $0$ and $I$ are Jordan all-derivable points
of the triangular algebra. In \cite{submit}, the authors study some
derivable mappings in the generalized matrix algebra $\mathcal{A}$,
and show that $0$, $P$ and $I$ are Jordan all-derivable points,
where $P$ is the standard non-trivial idempotent. In \cite{ZHUJUN},
Zhao and Zhu prove that every element in the algebra of all $n\times
n$ upper triangular matrices over the complex field $\mathbb{C}$ is a Jordan all-derivable point. In
Section 2, we give some general characterizations of Jordan
derivable mappings, which will be used to determine Jordan
all-derivable points for some general bimodules.

Let $\mathcal{A}$ be a unital algebra and $\mathbb{N}$ be the set of non-negative integers. A sequence of mappings
$\{d_i\}_{i\in\mathbb{N}}\in L(\mathcal{A})$ with $d_0=I_\mathcal{A}$ is called a
\textit{higher derivation} if $d_{n}(AB)=\sum_{i+j=n}d_i(A)d_j(B)$
for all $A,B\in \mathcal{A}$; it is called a \textit{Jordan higher
derivation} if
$d_{n}(AB+BA)=\sum_{i+j=n}(d_i(A)d_j(B)+d_i(B)d_j(A))$ for all
$A,B\in \mathcal{A}$. With the development of derivations, the study
of higher and Jordan higher derivations has attracted much attention
as an active subject of research in operator algebras, and the local
action problem ranks among in the list. A sequence of mappings
$\{d_i\}_{i\in\mathbb{N}}\in L(\mathcal{A})$ with $d_0=I_\mathcal{A}$ is called
\textit{Jordan higher derivable at $C\in \mathcal{A}$} if
$d_{n}(AB+BA)=\sum_{i+j=n}(d_i(A)d_j(B)+d_i(B)d_j(A))$ for all
$A,B\in \mathcal{A}$ with $AB=C$. An element $C\in \mathcal{A}$ is
called a \textit{Jordan higher all-derivable point} if every
sequence of Jordan higher derivable mappings at $C$ is a Jordan
higher derivation. In Section 3, we generalize the results in
Section 2 to the case of Jordan higher derivable mappings.
Meanwhile, we find the connection between Jordan all-derivable
points (all-derivable points, S-Jordan all derivable points,
respectively) and Jordan higher all-derivable points (higher
all-derivable points, S-Jordan higher all-derivable points,
respectively). We also discuss the automatic continuity property of
(Jordan) higher derivations.

Let $X$ be a complex Banach space and $B(X)$ be the set of all
bounded linear operators on $X$. For any non-empty subset
$L\subseteq X$, $L^\perp$ denotes its annihilator, that is,
$L^\perp=\{f\in X^*: f(x)=0~\mathrm{for}~\mathrm{all}~x\in L\}$. By
a \textit{subspace lattice} on $X$, we mean a collection
$\mathcal{L}$ of closed subspaces of $X$ with (0) and $X$ in
$\mathcal{L}$ such that for every family $\{M_r\}$ of elements of
$\mathcal{L}$, both $\cap M_r$ and $\vee M_r$ belong to
$\mathcal{L}$. For a subspace lattice $\mathcal{L}$ of $X$, let
alg$\mathcal{L}$ denote the algebra of all operators in $B(X)$ that
leave members of $\mathcal{L}$ invariant. A totally ordered subspace
lattice is called a \textit{nest}. If $\mathcal{L}$ is a nest, then
alg$\mathcal{L}$ is called a \textit{nest algebra}, see \cite{NEST}
for more on nest algebras. When $X$ is a separable Hilbert space
over the complex field $\mathbb{C}$, we change it to $H$. In a
Hilbert space, we disregard the distinction between a closed
subspace and the orthogonal projection onto it. An immediate but
noteworthy application of our main result shows that for a nest
$\mathcal{N}$ on a Banach $X$ with the associated nest algebra
$alg\mathcal{N}$, if there exists a non-trivial element in
$\mathcal{N}$ which is complemented in $X$, then every $C\in
alg\mathcal{N}$ is a Jordan all-derivable point of
$L(alg\mathcal{N}, B(X))$ and a Jordan higher all-derivable point of
$L(alg\mathcal{N})$.

\section{Jordan derivable mappings}\

We start with Peirce decomposition of algebras and its bimodules.

Let $\mathcal{A}$ be a unital algebra and $\mathcal{M}$ be a unital
$\mathcal{A}$-bimodule.  For any idempotent $E_1\in \mathcal{A}$,
let $E_2=I-E_1$. For $i,j\in\{1,2\},$ define
$\mathcal{A}_{ij}=E_i\mathcal{A}E_j$, which gives the Peirce
decomposition of $\mathcal{A}: A=A_{11}+A_{12}+A_{21}+A_{22}.$
Similarly, we define $\mathcal{M}_{ij}=E_i\mathcal{M}E_j$. We say
$\mathcal{A}_{ij}$ is \textit{left faithful} with respect to
$\mathcal{M}$ if for any $M\in \mathcal{M}$, the condition
$M\mathcal{A}_{ij}=\{0\}$  implies $ME_i=0$ and $\mathcal{A}_{ij}$
is \textit{right faithful} with respect to $\mathcal{M}$ if the
condition $\mathcal{A}_{ij}M=\{0\}$  implies $E_jM=0.$ We say
$\mathcal{A}_{ij}$ is \textit{faithful} with respect to
$\mathcal{M}$ if it is both left faithful and right faithful. In
this paper, we will always use the notations $P=E_1$ and
$Q=E_2=I-E_1$ for convenience.

In this section, we will assume $\mathcal{A}$ is a unital
algebra over a  field $\mathbb{F}$ of characteristic not equal to $2$ and $|\mathbb{F}|\geq4$,
$\mathcal{M}$ is a unital $\mathcal{A}$-bimodule and $\mathcal{A}$
has a non-trivial idempotent $P=E_1\in \mathcal{A}$ such that the
corresponding Peirce decomposition has the following property: Every
element of $A_{11}$ is a linear combination of invertible elements
of $A_{11}$ and every element of $A_{22}$ is a linear combination of
invertible elements of $A_{22}$. Algebras satisfying these
assumptions include all finite-dimensional unital algebras over an
algebraically closed field and all unital Banach algebras.

For any $A,B\in \mathcal{A}$, define $A\circ B=AB+BA$; similarly, for any
$A\in \mathcal A$ and $M\in \mathcal M$, define $A\circ M=AM+MA.$
For $A,B, D, E\in \mathcal{A}$, we say any $\delta \in L(\mathcal A, \mathcal M)$
\textit{differentiates} $A\circ B$ if $\delta(A\circ B)=\delta(A)\circ B+A\circ \delta(B)$
; we say $\delta $ \textit{differentiates} $A\circ B+ C\circ D$ if
$\delta (A\circ B+ C\circ D)=\delta (A)\circ B+A\circ \delta (B)+\delta (D)\circ E+D\circ \delta (E)$.
We see that a mapping $\delta\in L(\mathcal{A},\mathcal{M})$ is a Jordan derivation if and only if
$\delta $ differentiates $A\circ B$ for all
$A,B\in \mathcal{A}$, and $\delta$ is Jordan derivable at $C\in \mathcal{A}$
if and only if $\delta $ differentiates $A\circ B$ for all
$A,B\in \mathcal{A}$ with $AB=C$.

The following proposition is elementary, we omit the proof.

\begin{proposition}\label{0000}
Let $\mathcal{V}$ be a vector space over a field $\mathbb{F}$ with
$|\mathbb{F}|>n$.  For any fixed $v_i\in \mathcal{V}, i=0, 1, \cdots , n$, define
$p(t)=\sum_{i=0}^n v_i t^i $ for $t\in \mathbb{F}$. If $p(t)=0$ has
at least $n+1$ distinct solutions in $\mathbb{F}$, then $v_i=0, \ i=0, 1, \cdots , n$.
\end{proposition}

A simple application of Proposition $2.1$ yields the following proposition, which will be used repeatedly in this paper.

\begin{proposition}\label{0001}
Suppose $A, B, D, E, K, L \in \mathcal{A}$ and $\delta\in L(\mathcal{A},\mathcal{M})$.

(a) If $\delta $ differentiates $(tA+B)\circ (tD+E)$ for at least three $t\in \mathbb{F}$, then
$\delta $ differentiates $A\circ D$,  $B\circ E$, and $ A\circ E+B\circ D$; in particular, if $A=0$ then $\delta $
differentiates $B\circ D$.

(b) If $\delta $ differentiates  $A\circ (tD+E)+B\circ (tK+L)$ for at least two $t\in \mathbb{F}$, then
 $\delta $ differentiates  $A\circ D+B\circ K$ and  $A\circ E+B\circ L$.
\end{proposition}

Now we characterize Jordan-derivable mappings in terms of Peirce decomposition as follows.

\begin{theorem}\label{0002}
For any $C\in\mathcal{A}$ such that $C_{21}=0$, if $\Delta\in
L(\mathcal{A},\mathcal{M})$ is  Jordan-derivable at $C$, then there
exists a $\delta\in L(\mathcal{A},\mathcal{M})$ such that
$\Delta-\delta$ is an inner derivation and the following hold:

(a)~~$\delta(P)A_{12}=A_{12}\delta(Q)$ for any $A_{12}\in \mathcal{A}_{12}$.\

(b)~~$A_{12}\delta(A_{12})=\delta(A_{12})A_{12}=0$ for any $A_{12}\in \mathcal{A}_{12}$.\

(c)~~$\delta(\mathcal{A}_{11})\subset \mathcal{M}_{11}$, $\delta(\mathcal{A}_{22})\subset \mathcal{M}_{22}$.

If $\mathcal{A}_{12}$ is left faithful, then

(d)~~$\delta(P)\in C(\mathcal{A}_{11},\mathcal{M})$.

(e)~~$\delta|_{\mathcal{A}_{11}}$ is a generalized derivation from $\mathcal{A}_{11}$ to $\mathcal{M}_{11}$.

If $\mathcal{A}_{12}$ is right faithful, then

(f)~~$\delta(Q)\in C(\mathcal{A}_{22},\mathcal{M})$.

(g)~~$\delta|_{\mathcal{A}_{22}}$ is a generalized derivation from $\mathcal{A}_{22}$ to $\mathcal{M}_{22}$.
\end{theorem}

\begin{proof}
Let $M=P\Delta(Q)Q-Q\Delta(Q)P$ and define
$\delta(A)=\Delta(A)-(MA-AM)$ for every $A\in \mathcal{A}$.  Then
$\delta$ is Jordan-derivable at any $G\in \mathcal{A}$ if and only
if $\Delta$ is Jordan-derivable at $G$; moreover $\delta(Q)\in
\mathcal{M}_{11}+\mathcal{M}_{22}$ by direct computation. Write $C=
C_{11}+C_{12}+C_{22}$. Fix any $A_{11}\in\mathcal{A}_{11}$ that is
invertible in $\mathcal{A}_{11}$ with $A_{11}^{-1}\in
\mathcal{A}_{11}$ and  $Z_{22}, W_{22}\in \mathcal{A}_{22}$ such
that $Z_{22}W_{22}=C_{22}$. Note that we can take any $W_{22}$ that
is invertible in $\mathcal{A}_{22}$ with $W_{22}^{-1}\in
\mathcal{A}_{22}$ and $Z_{22}=C_{22}W_{22}^{-1}$ to satisfy
$Z_{22}W_{22}=C_{22}$. For any $0\neq t\in \mathbb{F}$,  $s\in \mathbb{F}$,  and
$A_{12}\in \mathcal{A}_{12}$, a routine computation yields
$$[A_{11}+t(sA_{11}A_{12}+Z_{22})][(A_{11}^{-1}C-sA_{12}W_{22})+t^{-1}W_{22}]=C.$$
Since $\delta$ is Jordan derivable at C, $\delta $ differentiates
$$[A_{11}+t(sA_{11}A_{12}+Z_{22})]\circ [(A_{11}^{-1}C-sA_{12}W_{22})+t^{-1}W_{22}].$$
Thus $\delta $ differentiates
$[A_{11}+t(sA_{11}A_{12}+Z_{22})]\circ [t(A_{11}^{-1}C-sA_{12}W_{22})+W_{22}].$
By Proposition \ref{0001}$(a)$, we get $(i)$  $\delta $ differentiates $A_{11}\circ W_{22}$, \ \
$(ii)$   $\delta $ differentiates $(sA_{11}A_{12}+Z_{22})\circ (A_{11}^{-1}C-sA_{12}W_{22})$, and
$(iii)$   $\delta $ differentiates  $A_{11}\circ (A_{11}^{-1}C-sA_{12}W_{22})+(sA_{11}A_{12}+Z_{22})\circ W_{22}.$

By $(i)$, we get
\begin{eqnarray}\delta (A_{11})\circ W_{22}+A_{11}\circ \delta (W_{22})=\delta (A_{11}\circ W_{22})=0 \label{2000}
\end{eqnarray}

By $(ii)$ and Proposition \ref{0001}$(a)$, we have $\delta $  differentiates $(A_{11}A_{12})\circ (A_{12}W_{22})$, i.e.
\begin{eqnarray}
\delta (A_{11}A_{12})\circ (A_{12}W_{22}) +(A_{11}A_{12})\circ \delta (A_{12}W_{22})
=\delta [(A_{11}A_{12})\circ (A_{12}W_{22})]=0\label {2002}
\end{eqnarray}

By $(iii)$ and Proposition \ref{0001}$(b)$, we have
\begin{eqnarray*}
&0=\delta [A_{11}\circ (-A_{12}W_{22})+(A_{11}A_{12})\circ W_{22}]
=\delta (A_{11})\circ (-A_{12}W_{22}) +A_{11}\circ \delta (-A_{12}W_{22}) \\
&~~~~~~~~~~~~~~~~~~~~~~~~~~~~~~~~~~~~~~~~~~~~~~~~~~~~~~~+\delta (A_{11}A_{12})\circ W_{22}+(A_{11}A_{12})\circ \delta (W_{22})
\end{eqnarray*}

Thus
\begin{eqnarray}
\delta (A_{11}A_{12})\circ W_{22}+(A_{11}A_{12})\circ \delta (W_{22})
-\delta (A_{11})\circ (A_{12}W_{22}) - A_{11}\circ \delta (A_{12}W_{22})=0 \label {2003}
\end{eqnarray}
Since $\delta (Q)\in \mathcal M_{11}+\mathcal M_{22}$, $A_{11}\circ \delta (Q)\in \mathcal M_{11}$.
Setting $W_{22}=Q$ in Eq. $(2.1)$ gives
\begin{eqnarray*} \delta (A_{11})\circ Q+A_{11}\circ \delta (Q)=0
\end{eqnarray*}
Thus
$A_{11}\circ \delta (Q)=\delta (A_{11})\circ Q=0$. It follows $\delta (A_{11})Q=Q\delta (A_{11})=0.$
Hence $\delta (A_{11})\in \mathcal M_{11},$ and $\delta (A_{11})\circ W_{22}=0$. By Eq. $(2.1)$ again, we get
$A_{11}\circ \delta (W_{22})=0$; in particular, $P\circ \delta (W_{22})=0$.
It follows that $\delta (W_{22})\in \mathcal M_{22}$,
 which proves $(c)$.

Taking $A_{11}=P$ in Eq. $(2.3)$ yields
\begin{eqnarray}
\delta (A_{12})\circ W_{22}+A_{12}\circ \delta (W_{22})
-\delta (P)\circ (A_{12}W_{22}) - P\circ \delta (A_{12}W_{22})=0 \label {2004}
\end{eqnarray}
Multiplying $P$ from both sides of Eq. $(2.4)$ gives
$P \delta (A_{12}W_{22})P=0$. In particular,
\begin{eqnarray}
P \delta (A_{12})P=0 \label {2005}
\end{eqnarray}
Multiplying $P$ from the left of Eq. $(2.4)$ and applying Eq. $(2.5)$, we get
\begin{eqnarray}
P\delta (A_{12}) W_{22}+A_{12}\delta (W_{22})
-\delta (P)A_{12}W_{22} - P\delta (A_{12}W_{22})=0
\end{eqnarray}
Setting $W_{22}=Q$ in Eq. $(2.6)$ and combining with Eq. $(2.5)$ leads to
\begin{eqnarray}
A_{12}\delta (Q)=\delta (P)A_{12} \label {2006}
\end{eqnarray}
This proves $(a)$.

Taking $A_{11}=P$ and $W_{22}=Q$ in Eq. $(2.2)$, we get $A_{12}\circ \delta (A_{12})=0$, i.e.
\begin{eqnarray}
A_{12}\delta (A_{12})+\delta (A_{12})A_{12}=0 \label {2007}
\end{eqnarray}
Multiplying $P$ from the left of Eq. $(2.8)$ and applying Eq. $(2.5)$, yields $A_{12}\delta (A_{12})=0$; which gives
$\delta (A_{12})A_{12}=0$, when applied to Eq. $(2.8)$. This proves $(b)$.

Taking $W_{22}=Q$ in Eq. $(2.3)$ yields
\begin{eqnarray}
\delta (A_{11}A_{12})\circ Q+(A_{11}A_{12})\circ \delta (Q)
-\delta (A_{11})\circ A_{12} - A_{11}\circ \delta (A_{12})=0 \label {2008}
\end{eqnarray}
Multiplying $Q$ from both sides of Eq. $(2.9)$ gives
$Q\delta (A_{11}A_{12})Q=0$. In particular,
\begin{eqnarray}
Q \delta (A_{12})Q=0 \label {2009}
\end{eqnarray}
Multiplying $Q$ from the right of Eq. $(2.9)$ and applying Eq. $(2.10)$ gives
\begin{eqnarray*}
\delta (A_{11}A_{12})Q+A_{11}A_{12}\delta (Q)
-\delta (A_{11})A_{12} - A_{11}\delta (A_{12})Q=0
\end{eqnarray*}
Combining this with Eq. $(2.5)$ yields
\begin{eqnarray}
\delta (A_{11}A_{12})Q=
\delta (A_{11})A_{12} + A_{11}\delta (A_{12}) - A_{11}A_{12}\delta (Q) \label {2010}
\end{eqnarray}
Replacing $A_{11}$ with $A_{11}U_{11}$ in Eq. $(2.11)$ gives
\begin{eqnarray}
\delta (A_{11}U_{11}A_{12})Q=
\delta (A_{11}U_{11})A_{12} + A_{11}U_{11}\delta (A_{12}) - A_{11}U_{11}A_{12}\delta (Q) \label {2011}
\end{eqnarray}
On the other hand, applying Eq. $(2.11)$ twice gives
\begin{align}
 \delta (A_{11}U_{11}A_{12})Q &=
A_{11}\delta (U_{11}A_{12}) + \delta (A_{11})U_{11}A_{12}  - A_{11}U_{11}A_{12}\delta (Q) \nonumber \\
& = A_{11}\delta (U_{11}A_{12})Q + \delta (A_{11})U_{11}A_{12} - A_{11}U_{11}A_{12}\delta (Q) \nonumber \\
& = A_{11}[\delta (U_{11})A_{12} + U_{11}\delta (A_{12}) - U_{11}A_{12}\delta (Q)]  \nonumber \\
&~~~~ + \delta (A_{11})U_{11}A_{12} - A_{11}U_{11}A_{12}\delta (Q) \label {2012}
\end{align}
By Eqs. $(2.12)$, $(2.13)$, and $(2.7)$, we have
\begin{eqnarray*}
\delta (A_{11}U_{11})A_{12}=
[\delta (A_{11})U_{11} + A_{11}\delta (U_{11}) - \delta (P)A_{11}U_{11}]A_{12}
\end{eqnarray*}
If $\mathcal A_{12}$ is left faithful,
\begin{eqnarray}
\delta (A_{11}U_{11})=
\delta (A_{11})U_{11} + A_{11}\delta (U_{11}) - \delta (P)A_{11}U_{11} \label {2013}
\end{eqnarray}
Taking $U_{11}=P$ in Eq. $(2.14)$ gives $A_{11}\delta (P) = \delta (P)A_{11}$, that is,
$\delta (P)\in C(\mathcal A_{11}, \mathcal M)$. This proves $(d)$ and now $(e)$ follows directly
from Eq. $(2.14)$.

Since $\delta(P)A_{12}=A_{12}\delta(Q)$ for any $A_{12}\in \mathcal{A}_{12}$, we
have $A_{12}\delta(Q)A_{22}=\delta(P)A_{12}A_{22}=A_{12}A_{22}\delta(Q)$, then faithfulness
of $\mathcal{A}_{12}$ leads to $\delta(Q)A_{22}=A_{22}\delta(Q)$, that is,
 $\delta(Q)\in C(\mathcal{A}_{22},\mathcal{M})$. This proves $(f)$.

By Eqs. $(2.6)$ and $(2.10)$,
\begin{eqnarray}
P\delta (A_{12}W_{22})=\delta (A_{12}) W_{22}+A_{12}\delta (W_{22})
-\delta (P)A_{12}W_{22}
\end{eqnarray}
Replacing $W_{22}$ with $V_{22}W_{22}$ in Eq. $(2.15)$ gives
\begin{eqnarray}
P\delta (A_{12}V_{22}W_{22})=\delta (A_{12}) V_{22}W_{22}+A_{12}\delta (V_{22}W_{22})
-\delta (P)A_{12}V_{22}W_{22}
\end{eqnarray}
On the other hand, applying Eq. $(2.15)$ twice gives
\begin{align}
P\delta (A_{12}V_{22}W_{22})&=\delta (A_{12}V_{22}) W_{22}+A_{12}V_{22}\delta (W_{22})
-\delta (P)A_{12}V_{22}W_{22} \nonumber \\
&=P\delta (A_{12}V_{22}) W_{22}+A_{12}V_{22}\delta (W_{22})
-\delta (P)A_{12}V_{22}W_{22} \nonumber \\
&=[\delta (A_{12}) V_{22}+A_{12}\delta (V_{22})
-\delta (P)A_{12}V_{22}]W_{22} \nonumber \\
&~~~~+A_{12}V_{22}\delta (W_{22})
-\delta (P)A_{12}V_{22}W_{22}
\end{align}
By Eqs. $(2.16)$, $(2.17)$, and $(2.7)$,
\begin{eqnarray*}
A_{12}\delta (V_{22}W_{22})=A_{12}[\delta (V_{22})W_{22}+V_{22}\delta (W_{22})
-\delta (Q)V_{22}W_{22}]
\end{eqnarray*}
Since $\mathcal A_{12}$ is left faithful,
\begin{eqnarray}
\delta (V_{22}W_{22})=\delta (V_{22})W_{22}+V_{22}\delta (W_{22})
-\delta (Q)V_{22}W_{22}
\end{eqnarray}
This proves $(g)$.
\end{proof}

Suppose $\mathcal{B}$ is an algebra containing $\mathcal{A}$ and shares the same identity with $\mathcal{A}$, then
$\mathcal{B}$ is an $\mathcal{A}$-bimodule with respect to the
multiplication and addition of $\mathcal{B}$. Let $\mathcal{T}_{\mathcal{A}}= \{A\in \mathcal A: A_{21}=0\}$.
 The following proposition is contained in  \cite[Theorem 3.3]{PAN}, we include a proof here for completeness.

\begin{proposition}\label{2.4}
Suppose $\mathcal{A}_{12}$ is faithful to $\mathcal{B}$, $C(\mathcal{T}_{\mathcal{A}},\mathcal{B})=\mathbb{F}I$, and $B\in \mathcal B$.
If $T_{12}BT_{12}=0$ for every $T_{12}\in \mathcal{A}_{12}$, then $QBP=0$.
\end{proposition}
\begin{proof}
Suppose $T_{12}BT_{12}=0$ for every $T_{12}\in \mathcal{A}_{12}$. For any non-zero $A_{12},T_{12}\in\mathcal{A}_{12}$, we have $T_{12}BT_{12}=0$, $A_{12}BA_{12}=0$ and $(A_{12}+T_{12})B(A_{12}+T_{12})=0$. It follows that
\begin{eqnarray}
A_{12}BT_{12}+T_{12}BA_{12}=0.\label{2023}
\end{eqnarray}
For any $A_{11}\in \mathcal{A}_{11}$, replacing $A_{12}$ in Eq. $(2.19)$ with $A_{11}A_{12}$ gives
\begin{eqnarray}
A_{11}A_{12}BT_{12}+T_{12}BA_{11}A_{12}=0.\label{2024}
\end{eqnarray}
Multiplying $A_{11}$ from the left of Eq. $(2.19)$ gives
\begin{eqnarray}
A_{11}A_{12}BT_{12}+A_{11}T_{12}BA_{12}=0.\label{2025}
\end{eqnarray}
By Eq. (\ref{2024}) and Eq. (\ref{2025}), we have
\begin{eqnarray*}
T_{12}BA_{11}A_{12}=A_{11}T_{12}BA_{12}.
\end{eqnarray*}
Since $A_{12}$ is arbitrary and $\mathcal{A}_{12}$ is faithful, we have
\begin{eqnarray}
T_{12}BA_{11}=A_{11}T_{12}BP.\label{2026}
\end{eqnarray}
Similarly, we have
\begin{eqnarray}
A_{22}QBT_{12}=QBT_{12}A_{22}.\label{2027}
\end{eqnarray}
Let $\widetilde{B}=T_{12}BP-QBT_{12}$. It follows from Eqs. (\ref{2023}),
(\ref{2026}) and  (\ref{2027}) that $\widetilde{B}$ commutes with $A_{12}$,
$A_{11}$ and $A_{22}$, that is, $\widetilde{B}\in C(\mathcal{T}_{\mathcal{A}},\mathcal{B})$.
Hence there exists a $k\in \mathbb{F}$ such that $\widetilde{B}=kI$. It follows
$T_{12}BP=kP$. Now $T_{12}BT_{12}=0$ leads to $kT_{12}=0$. Hence
$k=0$ and $T_{12}BP=0$. Since $T_{12}$ is arbitrary and
$\mathcal{A}_{12}$ is faithful, we have $QBP=0$.
\end{proof}

\begin{theorem}\label{2.5}
Suppose $\mathcal{A}_{12}$ is faithful to $\mathcal{B}$ and $C(\mathcal{T}_{\mathcal{A}},\mathcal{B})=\mathbb{F}I$.
If $\delta\in L(\mathcal{A},\mathcal{B})$ is Jordan derivable at $C\in \mathcal{T}_{\mathcal{A}}$ then
$\delta |_{\mathcal{T}_{\mathcal{A}}}$ is a derivation from $\mathcal{T}_{\mathcal{A}}$ to $\mathcal{B}$.
\end{theorem}
\begin{proof}
Substracting an
inner derivation from $\delta$ if necessary, we can assume $\delta$
satisfies the properties of Theorem $2.3$.
Thus, for any $A_{12}$ and $T_{12}$ in $\mathcal{A}_{12}$, we have $\delta(A_{12})A_{12}=0$,
$\delta(T_{12})T_{12}=0$ and
$\delta(A_{12}+T_{12})(A_{12}+T_{12})=0$. It follows that
$\delta(A_{12})T_{12}+\delta(T_{12})A_{12}=0.$ Multiplying $T_{12}$
from the left we obtain $T_{12}\delta(A_{12})T_{12}=0$. Since
$T_{12}$ is arbitrary, $Q\delta(A_{12})P=0$, by Proposition $2.4$. This, together with Eqs.
$(2.5)$ and  $(2.10)$, yields $\delta(A_{12})\in
\mathcal{B}_{12}$.

For any $A_{11}\in \mathcal A_{11}$ and $A_{22}\in \mathcal A_{22}$, by Theorem $2.3$
$\delta (A_{11})\in \mathcal B_{11}$ and $\delta (A_{22})\in \mathcal B_{22}$.

Since $\delta(P)\in\mathcal{B}_{11}$ and
$\delta(Q)\in\mathcal{B}_{22}$, by Theorem $2.3$
$\delta(I)=\delta(P)+\delta(Q)$ commutes with $\mathcal{A}_{11}$,
$\mathcal{A}_{12}$ and $\mathcal{A}_{22}$, whence $\delta(I)\in
C(\mathcal{T}_{\mathcal{A}},\mathcal{B})$. Thus $\delta(I)=\lambda I$. By the fact
that $\delta$ is Jordan derivable at $C$, we have
$\delta(IC+CI)=\delta(I)C+I\delta(C)+C\delta(I)+\delta(C)I$, which
implies $\lambda C=0$. If $C\neq0$, $\lambda=0$. Hence
$\delta(P)=\delta(Q)=0$. If $C=0$, then the fact that
$A_{12}A_{11}=0$ holds for every $A_{11}\in \mathcal{A}_{11}$ and
$A_{12}\in \mathcal{A}_{12}$ implies that
$\delta(A_{11}A_{12})=\delta(A_{12})A_{11}+A_{12}\delta(A_{11})+A_{11}\delta(A_{12})+\delta(A_{11})A_{12}=A_{11}\delta(A_{12})+\delta(A_{11})A_{12}$,
which together with faithfulness of $\mathcal{A}_{12}$ leads to
$\delta(A_{11}U_{11})=\delta(A_{11})U_{11}+A_{11}\delta(U_{11})$ for
every $A_{11}, U_{11}\in \mathcal{A}_{11}$. Comparing with
Eq. $(2.14)$, we have that $\delta(P)=\delta(Q)=0$.

To see $\delta |_{\mathcal{T}_{\mathcal{A}}}$ is a derivation, it suffices to show that for any $A_{ij}, A_{kl}\in \mathcal{T}_{\mathcal{A}}$
$$\delta(A_{ij}A_{kl})=\delta(A_{ij})A_{kl}+A_{ij}\delta(A_{kl}).$$
We will label each case as Case $(ij,kl)$. Since $\delta(A_{11})\in\mathcal{B}_{11}$, $\delta(A_{12})\in\mathcal{B}_{12}$,
and $\delta(A_{22})\in\mathcal{B}_{22}$, we only need to check cases for $j=k$. There are only 4 cases.

Case $(11,11)$ follows from Eq. $(2.14)$.

Case $(11,12)$ follows from Eq. $(2.11)$.

Case $(12,22)$ follows from Eq. $(2.15)$.

Case $(22,22)$ follows from Eq. $(2.18)$.
\end{proof}

\begin{corollary}\label{0002}
Suppose $\mathcal{A}_{12}$ is faithful to $\mathcal{B}$, $\mathcal{A}_{21}=\{0\}$, and
$C(\mathcal{A},\mathcal{B})=\mathbb{F}I$. If $\delta\in L(\mathcal{A},\mathcal{B})$ is Jordan derivable
at $C\in \mathcal A$ then $\delta $ is a derivation. In particular, every $C\in \mathcal A$ is a Jordan all-derivable point
of $L(\mathcal A, \mathcal B)$ and every Jordan derivation is a derivation.
\end{corollary}

As a consequence of Corollary \ref{0002}, similar to \cite[Theorem 4.4]{PAN} we have

\begin{theorem}\label{0003}
Let $\mathcal{L}$ be a subspace lattice on a Banach space $X$ and
$\mathcal{A}=alg\mathcal{L}$. Suppose there exists a non-trivial
idempotent $P\in \mathcal{A}$ such that $ran(P)\in \mathcal{L}$ and
$PB(X)(I-P)\subseteq \mathcal{A}$. If $\delta\in L(\mathcal{A},B(X))$ is Jordan derivable
at $C\in \mathcal A$ then $\delta $ is a derivation. In particular, every $C\in \mathcal A$ is a Jordan all-derivable point
of $L(\mathcal A, B(X))$.
\end{theorem}

\begin{proof}
We will apply Corollary \ref{0002} with $\mathcal{B}=B(X)$. Let
$Q=I-P$. The condition $ran(P)\in \mathcal{L}$ implies
$\mathcal{A}_{21}=Q\mathcal{A}P=\{0\}$. The condition
$PB(X)Q\subseteq \mathcal{A}$ implies $\mathcal{A}_{12}=PB(X)Q$ is
faithful. To see $C(\mathcal{A},B(X))=\mathbb{C}I$, take any $B\in
C(\mathcal{A},B(X))$, from $BP=PB$ we get $PBQ=0$. From $BQ=QB$, we
have $QBP=0$. Thus $B=B_{11}+B_{22}$. For any $x\in ran(P)$ and
$f\in X^*$, $x\otimes f Q\in \mathcal{A}_{12}$. It follows from
$Bx\otimes f Q=x\otimes f Q B$ that $B_{11}x\otimes f Q=x\otimes f Q
B_{22}$, which leads to $B_{11}x\in \mathbb{C}x$. Since $x\in
ran(P)$ is arbitrary, it follows $B_{11}=kP$ for some $k\in
\mathbb{C}$. Hence $x\otimes f(kQ-B_{22})=0$, and we have
$B_{22}=kQ$ and $B=kI$. Now the conclusion follows from Corollary
\ref{0002}.
\end{proof}

As an immediate but noteworthy application of Theorem \ref{0003}, we
have the following corollary which generalizes the main result in
\cite{ZHUJUN}.

\begin{corollary}
Let $\mathcal{N}$ be a nest on a Banach space $X$ and
$\mathcal{A}=alg\mathcal{N}$. Suppose there exists a non-trivial
idempotent $P\in \mathcal{A} $ such that $ran(P) \in \mathcal{N}$.
If $\delta\in L(\mathcal{A},B(X))$ is Jordan derivable
at $C\in \mathcal A$ then $\delta $ is a derivation.
In particular, every $C\in \mathcal A$ is a Jordan all-derivable point
of $L(\mathcal A, B(X))$.
\end{corollary}
\begin{proof}
Let $Q=I-P$.  Then $PB(X)Q\subseteq \mathcal{A}$. Now applying Theorem \ref{0003} completes the proof.
\end{proof}
For an algebra $\mathcal{A}$ and a left $\mathcal{A}$-module $\mathcal{M}$,
 we call a subset $\mathcal{B}$ of $\mathcal{A}$ \emph{separates} $\mathcal{M}$
 if for every $M\in \mathcal{M}$, $\mathcal{B}M={0}$ implies $M=0$.
Let $[\mathcal A_{11}, \mathcal A_{11}]=\{A_{11}B_{11}-B_{11}A_{11}: \ A_{11}, B_{11}\in \mathcal A_{11}\}$.

\begin{theorem} \label{0008}
Suppose $\mathcal{A}_{12}$ is faithful to $\mathcal{B}$, $C(\mathcal{T}_{\mathcal{A}},\mathcal{B})=\mathbb{F}I$, and
$[\mathcal A_{11}, \mathcal A_{11}]$ separates $\mathcal B_{12}$.  If $\delta \in L(\mathcal A, \mathcal B)$ is Jordan derivable at some
$C\in \mathcal A_{11}+\mathcal A_{12}$ then $\delta $ is a derivation. In particular, every $C\in \mathcal A_{11}+\mathcal A_{12}$
is a Jordan all-derivable point of $L(\mathcal A, \mathcal B).$
\end{theorem}
\begin{proof} Let $C\in \mathcal A_{11}+\mathcal A_{12}$ and
$\delta\in L(\mathcal{A},\mathcal{B})$ be Jordan derivable at $C$.
Substracting an inner derivation from $\delta$ if necessary, we can assume $\delta$
satisfies the properties of Theorem $2.3$.
Let $Q=I-P$ then $QC=0$.
By Theorem $2.3$ and Proposition \ref{2.4}, $\delta (\mathcal A_{11})\subseteq \mathcal B_{11}$,
$\delta (\mathcal A_{22})\subseteq \mathcal B_{22}$, and $\delta (\mathcal A_{12})\subseteq \mathcal B_{12}$;
moreover, $\delta (I)=\delta (P)=\delta (Q)=0.$
For any $t\in \mathbb F$, $A_{11}\in \mathcal A_{11}$ that is invertible in $\mathcal A_{11}$ with
$A_{11}^{-1}\in \mathcal A_{11}$, and $A_{21}\in \mathcal A_{21}$, clearly $A_{11}(A_{11}^{-1}C+tA_{21})=C$.
Since $\delta $ is Jordan derivable at $C$, $\delta $ differentiates $A_{11}\circ (A_{11}^{-1}C+tA_{21})$.
By Proposition $2.2$,
\begin{eqnarray}
\delta (A_{11}\circ A_{21}) = \delta (A_{11})\circ A_{21}+A_{11}\circ \delta (A_{21})
\end{eqnarray}
Multiplying $Q$ from the right of Eq. $(2.24)$ gives
\begin{eqnarray}
A_{11}\delta (A_{21})Q = \delta (A_{21}A_{11})Q
\end{eqnarray}
For any $U_{11}\in \mathcal A_{11}$, by Eq. $(2.25)$ we get
\begin{eqnarray*}
A_{11}U_{11}\delta (A_{21})Q = \delta (A_{21}A_{11}U_{11})Q
\end{eqnarray*}
On the other hand, applying Eq. $(2.25)$ twice gives
\begin{eqnarray*}
U_{11}A_{11}\delta (A_{21})Q = U_{11}\delta (A_{21}A_{11})Q = \delta (A_{21}A_{11}U_{11})Q
\end{eqnarray*}
It follows $[A_{11}U_{11}-U_{11}A_{11}]\delta (A_{21})Q=0$.
Since $[\mathcal A_{11}, \mathcal A_{11}]$ separates $\mathcal B_{12}$, $P\delta (A_{21})Q=0$.
Multiplying $Q$ from the left of Eq. $(2.25)$ gives $Q\delta (A_{21}A_{11})Q = 0$. In particular,
$Q\delta (A_{21})Q = 0$.
Multiplying $P$ from both sides of Eq. $(2.24)$ and setting $A_{11}=P$ leads to
$P\delta (A_{21})P = 0$. Thus $\delta (A_{21})\in \mathcal B_{21}$.

To see $\delta $ is a derivation, it suffices to show that for any $A_{ij}\in \mathcal A_{ij}, A_{kl}\in \mathcal A_{kl}$
$$\delta(A_{ij}A_{kl})=\delta(A_{ij})A_{kl}+A_{ij}\delta(A_{kl})$$
We will again label each case as Case $(ij,kl)$. Since $\delta(A_{ij})\in\mathcal{B}_{ij}$, for all $i, j=1, 2,$
we only need to check cases for $j=k$. There are 8 cases.

Case $(11,11)$ follows from Eq. $(2.14)$.

Case $(11,12)$ follows from Eq. $(2.11)$.

Case $(12,22)$ follows from Eq. $(2.15)$.

Case $(22,22)$ follows from Eq. $(2.18)$.

Case $(21,11)$ follows from Eq. $(2.24)$. \\
It remains to show Cases $(12, 21)$, $(21, 12)$, and $(22, 21)$.

For any $s, t \in \mathbb F$, a routine computation shows $(P+sA_{12})[t(A_{21}-sA_{12}A_{21})-sA_{12}+C+Q]=C$.
Since $\delta $ is Jordan derivable at $C$, $\delta $ differentiates
 $(P+sA_{12})\circ [t(A_{21}-sA_{12}A_{21})-sA_{12}+C+Q].$ By  Proposition $2.2$,
$\delta $ differentiates  $(P+sA_{12})\circ (A_{21}-sA_{12}A_{21})$. Applying   Proposition $2.2$ again , we see
 $\delta $ differentiates $P\circ (-A_{12}A_{21})+A_{12}\circ A_{21}$. Case $(11,11)$ implies $\delta $  differentiates $P\circ (-A_{12}A_{21})$,
It follows that $\delta $ differentiates $A_{12}\circ A_{21},$ i.e.
\begin{eqnarray}
\delta (A_{12}\circ A_{21}) = \delta (A_{12})\circ A_{21}+A_{12}\circ \delta (A_{21})
\end{eqnarray}
Multiplying $P$ from both sides of Eq. $(2.26)$ gives Case $(12, 21)$ and multiplying $Q$ from both sides of Eq. $(2.26)$ gives Case $(21, 12)$.

Applying Case $(21, 12)$, we obtain
\begin{eqnarray}
\delta (A_{22}A_{21}A_{12})=\delta (A_{22}A_{21})A_{12}+A_{22}A_{21}\delta (A_{12})
\end{eqnarray}

Using Cases $(22, 22)$ and $(21, 12)$, we have
\begin{eqnarray} & \delta (A_{22}A_{21}A_{12}) =\delta (A_{22})A_{21}A_{12}+A_{22}\delta (A_{21}A_{12})\nonumber \\
&~~~~~~~~~~= \delta (A_{22})A_{21}A_{12}+A_{22}\delta (A_{21})A_{12}+A_{22}A_{21}\delta (A_{12})
\end{eqnarray}

By $(2.27)$ and $(2.28)$,
$\delta (A_{22}A_{21})A_{12} =  \delta (A_{22})A_{21}A_{12}+A_{22}\delta (A_{21})A_{12}.$ Since $\mathcal A_{12}$ is faithful, we get
$\delta (A_{22}A_{21})= \delta (A_{22})A_{21}+A_{22}\delta (A_{21}),$ completing the proof of Case $(21, 11)$.
\end{proof}

\begin{corollary} Suppose $H$ is a Hilbert space and $C\in B(H)$ such that $\ker (C)\neq 0$ or $\ker (C^{\ast})\neq 0.$ If
$\delta \in L(B(H), B(H))$ is Jordan derivable at $C$  then $\delta $ is a derivation. In particular,
$C$ is a Jordan all-derivable point of $L(B(H), B(H)).$
\end{corollary}
\begin{proof}
If  $\ker (C^{\ast})\neq 0$, then there exists a proper orthogonal projection $P\in B(H)$ such that $ran(C)\subseteq PH$.
Let $Q=I-P$ then $QC=0$. Take $\mathcal A=\mathcal B =B(H)$, the one can check that all hypotheses of
Theorem $2.9$ are satisfied and the conclusions follow.

If $\ker (C)\neq 0$, we can define $\delta ^{\ast}\in L(B(H), B(H))$
 by $\delta ^{\ast}(A)=(\delta(A^{\ast}))^{\ast}$ for every $A\in B(H)$.
  Since $\delta$ is Jordan derivable at $C$, we have $\delta ^{\ast}$ is
  Jordan derivable at $C^{\ast}$. Now by the argument in the first paragraph
  we have $\delta ^{\ast}$ is a derivation, and in turn $\delta$ is a derivation.
   This completes the proof.
\end{proof}

\section{Jordan higher derivable mappings}\

In this section, we assume that $\mathcal{A}$ is an algebra over a field
$\mathbb{F}$ of characteristic zero. Before stating our main results
in this section, we first need a proposition that characterizes
Jordan higher derivations in terms of Jordan derivations. Since the
proof is similar to the proof of \cite[Theorem 2.5]{Higher}, we omit
it here.

\begin{proposition}\label{0005}
Let $\mathcal{A}$ be an algebra, $\{d_i\}_{i\in\mathbb{N}}$ be a
sequence of mappings on $\mathcal{A}$ with $d_0=I_\mathcal{A}$ and
$\{\delta_i\}_{i\in\mathbb{N}}$ be the a sequences of (Jordan)
derivations on $\mathcal{A}$ with $\delta_0=0$. If the following
recursive relation holds: \begin{eqnarray*}
nd_n=\sum_{k=0}^{n-1}\delta_{k+1}d_{n-1-k}
\end{eqnarray*}
for $n\geq1$, then $\{d_i\}_{i\in\mathbb{N}}$ is a (Jordan) higher derivation.
\end{proposition}

Let $\mathcal R(\mathcal A)$ be a relation on $\mathcal A$, i.e. $\mathcal R(\mathcal A)$ is a
nonempty subset of $\mathcal A \times \mathcal A$. We say $\delta\in L(\mathcal{A},\mathcal{M})$ is
\textit{derivable on} $\mathcal R(\mathcal A)$ if $\delta(AB)=\delta(A)B+A\delta(B)$ for all
$(A,B)\in \mathcal R(\mathcal{A})$. We say $\delta\in L(\mathcal{A},\mathcal{M})$ is
\textit{Jordan derivable on} $\mathcal R(\mathcal A)$ if $\delta(AB+BA)=\delta(A)B+A\delta(B)+\delta(B)A+B\delta(A)$ for all
$(A,B)\in \mathcal R(\mathcal{A})$. A sequence of mappings
$\{d_i\}_{i\in\mathbb{N}}\in L(\mathcal{A})$ with $d_0=I_\mathcal{A}$ is called
\textit{higher derivable on} $\mathcal R(\mathcal A)$  if
$d_{n}(AB)=\sum_{i+j=n}d_i(A)d_j(B)$ for all
$(A,B)\in \mathcal R(\mathcal{A})$. A sequence of mappings
$\{d_i\}_{i\in\mathbb{N}}\in L(\mathcal{A})$ with $d_0=I_\mathcal{A}$ is called
\textit{Jordan higher derivable on} $\mathcal R(\mathcal A)$  if
$d_{n}(AB+BA)=\sum_{i+j=n}(d_i(A)d_j(B)+d_i(B)d_j(A))$ for all
$(A,B)\in \mathcal R(\mathcal{A})$. We say $\mathcal R(\mathcal A)$ is \textit{(Jordan) derivational}
for $L(\mathcal{A},\mathcal{M})$ if every (Jordan) derivable
mapping on $\mathcal R(\mathcal A)$ is a (Jordan) derivation.
We say $\mathcal R(\mathcal A)$ is \textit{(Jordan) higher derivational}
for $L(\mathcal{A})$ if every (Jordan) higher derivable
mapping on $\mathcal R(\mathcal A)$ is a (Jordan) higher derivation.

\begin{remark}
\emph{The above definitions allow us to unify some
of the notions in the literature. For example,
in literature, there are two definitions of Jordan derivable
mappings, one is what we use in this paper (see for example \cite{chen, ZHUJUN} and
references therein), and the other
(see for example \cite{Hou, Hou3, ZHUJUN1}) is what we call S-Jordan
derivable mappings (S stands for standard). A mapping $\delta\in
L(\mathcal{A},\mathcal{M})$ is called a \textit{S-Jordan derivable
mapping at $C\in \mathcal{A}$} if
$\delta(AB+BA)=\delta(A)B+A\delta(B)+\delta(B)A+B\delta(A)$ for all
$A,B\in \mathcal{A}$ with $AB+BA=C$. An element $C\in \mathcal{A}$
is called a \textit{S-Jordan all-derivable point} if every S-Jordan
derivable mapping at $C$ is a Jordan derivation. A sequence of
mappings $\{d_i\}_{i\in\mathbb{N}}\in L(\mathcal{A})$ with $d_0=I_\mathcal{A}$ is
called \textit{S-Jordan higher derivable at $C\in \mathcal{A}$} if
$d_{n}(AB+BA)=\sum_{i+j=n}(d_i(A)d_j(B)+d_i(B)d_j(A))$ for all
$A,B\in \mathcal{A}$ with $AB+BA=C$. An element $C\in \mathcal{A}$
is called a \textit{S-Jordan higher all-derivable point} if every
sequence of S-Jordan higher derivable mappings at $C$ is a Jordan
higher derivation. The above two notions of Jordan derivable mappings at $C$ are special case of
Jordan derivable mappings on $\mathcal R(\mathcal A)$,
where $\mathcal R(\mathcal A)=\{(A, B)\in \mathcal A \times \mathcal A: AB=C\}$ and
$\mathcal R(\mathcal A)=\{(A, B)\in \mathcal A \times \mathcal A: AB+BA=C\}$, respectively.}
\end{remark}

\begin{theorem} \label{0006}
If $\mathcal{A}$ is an algebra such that $\mathcal R(\mathcal A)$ is (Jordan) derivational for $L(\mathcal{A})$,
then $\mathcal R(\mathcal A)$ is  (Jordan) higher derivational.
\end{theorem}

\begin{proof}
First, suppose $\mathcal R(\mathcal A)$ is Jordan derivational and
$\{d_i\}_{i\in\mathbb{N}}$ is a sequence of mappings in $L(\mathcal{A})$
Jordan higher derivable on $\mathcal R(\mathcal A)$. Let $\delta_1=d_1$ and
$\delta_n=nd_n-\sum_{k=0}^{n-2}\delta_{k+1}d_{n-1-k}$ for every
$n(\geq2)\in \mathbb{N}.$ We will show $\{\delta_i\}_{i\in\mathbb{N}}$ is a sequence
of Jordan derivations, and in turn $\{d_i\}_{i\in\mathbb{N}}$ is a Jordan higher
derivation by Proposition \ref{0005}. We prove by induction.

When $n=1$, since $\mathcal R(\mathcal A)$ is Jordan derivational, we have that $\delta_1$ is a Jordan derivation.

Now suppose $\delta_k$ is defined as above and is a Jordan
derivation for $k\leq n$. For $(A,B)\in \mathcal R(\mathcal{A})$, we
have
\begin{eqnarray*}
&&\delta_{n+1}(A\circ B)=(n+1)d_{n+1}(A\circ B)-\sum_{k=0}^{n-1}\delta_{k+1}d_{n-k}(A\circ B)\nonumber\\
&&~~~~~~~~~~~~~~~=(n+1)\sum_{k=0}^{n+1}\{d_k(A)\circ d_{n+1-k}(B)\}
-\sum_{k=0}^{n-1}\delta_{k+1}\sum_{l=0}^{n-k}\{d_l(A)\circ d_{n-k-l}(B)\}.\nonumber\\
\end{eqnarray*}
By induction we have
\begin{eqnarray*}
&&\delta_{n+1}(A\circ B)
=\sum_{k=0}^{n+1}kd_k(A)\circ d_{n+1-k}(B)+\sum_{k=0}^{n+1}d_k(A)\circ (n+1-k)d_{n+1-k}(B)\nonumber\\
&&~~~~~~~~~~~~~-\sum_{k=0}^{n-1}\sum_{l=0}^{n-k}\{\delta_{k+1}(d_l(A))\circ d_{n-k-l}(B)+d_l(A)\circ \delta_{k+1}(d_{n-k-l}(B))\}.\nonumber\\
\end{eqnarray*}
Set
\begin{eqnarray*}
&&K_1=\sum_{k=0}^{n+1}kd_k(A)\circ d_{n+1-k}(B)-\sum_{k=0}^{n-1}\sum_{l=0}^{n-k}\delta_{k+1}(d_l(A))\circ d_{n-k-l}(B),\\
&&K_2=\sum_{k=0}^{n+1}d_k(A)\circ (n+1-k)d_{n+1-k}(B)-\sum_{k=0}^{n-1}\sum_{l=0}^{n-k}d_l(A)\circ \delta_{k+1}(d_{n-k-l}(B)).\\
\end{eqnarray*}
Then $\delta_{n+1}(A\circ B)=K_1 +K_2$.
Let us compute $K_1$ and $K_2$. If we put $r=k+l$ in the summation $\sum_{k=0}^{n-1}\sum_{l=0}^{n-k}$, then we may write it as $\sum_{r=0}^{n}\sum_{0\leq k\leq r, k\neq n}$. Hence
\begin{eqnarray*}
K_1=\sum_{k=0}^{n+1}kd_k(A)\circ d_{n+1-k}(B)-\sum_{r=0}^{n}\sum_{0\leq k\leq r, k\neq n}\delta_{k+1}(d_{r-k}(A))\circ d_{n-r}(B).
\end{eqnarray*}
Putting $r+1$ instead of $k$ in the first summation, we have
\begin{eqnarray*}
&&K_1+\sum_{k=0}^{ n-1}\delta_{k+1}(d_{n-k}(A))\circ B\\
&&~~~=\sum_{r=0}^{n}(r+1)d_{r+1}(A)\circ d_{n-r}(B)-\sum_{r=0}^{n-1}\sum_{k=0}^r\delta_{k+1}(d_{r-k}(A))\circ d_{n-r}(B)\\
&&~~~=\sum_{r=0}^{n-1}\{(r+1)d_{r+1}(A)-\sum_{k=0}^r\delta_{k+1}(d_{r-k}(A))\}\circ d_{n-r}(B)+(n+1)d_{n+1}(A)\circ B.
\end{eqnarray*}
By our assumption $(r+1)d_{r+1}(A)=\sum_{k=0}^r\delta_{k+1}(d_{r-k}(A))$ for $r=0, \ldots, n-1,$ we obtain
\begin{eqnarray*}
&&K_1=(n+1)d_{n+1}(A)\circ B-\sum_{k=0}^{ n-1}\delta_{k+1}(d_{n-k}(A))\circ B=\delta_{n+1}(A)\circ B.\\
\end{eqnarray*}
Similary, we may deduce that
\begin{eqnarray*}
&&K_2=(n+1)A\circ d_{n+1}(B)-\sum_{k=0}^{ n-1}A\circ \delta_{k+1}(d_{n-k}(B))=A\circ \delta_{n+1}(B).\\
\end{eqnarray*}
Therefore, $\delta_{n+1}$ is Jordan derivable on $\mathcal R(\mathcal A)$.
Since $\mathcal R(\mathcal A)$ is Jordan derivational, we have that
$\delta_{n+1}$ is a Jordan derivation.

Similarly, we can prove the case when $\mathcal R(\mathcal A)$ is assumed to be derivational by changing ``$\ \circ $" to
the normal multiplication of $\mathcal A$.
\end{proof}

Recall that a mapping $\delta\in L(\mathcal{A},\mathcal{M})$ is called \textit{derivable at $C\in \mathcal{A}$}
if $\delta(AB)=\delta(A)B+A\delta(B)$ for all $A,B\in \mathcal{A}$ with $AB=C$. An element $C\in \mathcal{A}$ is called an \textit{all-derivable point} if every derivable mapping at $C$ is a derivation. A sequence of mappings $\{d_i\}_{i\in\mathbb{N}}\in L(\mathcal{A})$ with $d_0=I_\mathcal{A}$ is called \textit{higher derivable at $C\in \mathcal{A}$} if $d_{n}(AB)=\sum_{i+j=n}d_i(A)d_j(B)$ for all $A,B\in \mathcal{A}$ with $AB=C$. An element $C\in \mathcal{A}$ is called a \textit{higher all-derivable point} if every sequence of higher derivable mappings at $C$ is a higher derivation.

\begin{remark}
\emph{Several authors (see for example \cite{Hou5,Hou6, Hou4,
LI7,LI5,LU1, Hou2,Zhu2,Zhu3,Zhu4,ZHUJUN2})  investigate derivable
mappings at 0, invertible element, left (right) separating point,
non-trivial idempotent, and the unit $I$ on certain algebras. By
Theorem \ref{0006}, we can generalize these results to the higher
derivation case. Many authors also study (S-)Jordan derivable
mappings (see for example \cite{ Hou, chen,Hou3,ZHUJUN1,ZHUJUN}) at
these points. Theorems \ref{0006}
also allow us to generalize these results to the (S-)Jordan higher derivation
case. }
\end{remark}

Combining Theorem \ref{0006} with Corollary \ref{0002}, we have

\begin{corollary}\label{0007}
Suppose $\mathcal{A}$, $\mathcal{B}$ are as in Corollary \ref{0002} with $\mathcal{B}=\mathcal{A}$.
Then every $C\in \mathcal{A}$
is a Jordan higher all-derivable point.
\end{corollary}

Combining Theorem \ref{0006} with Theorem \ref{0003}, we have

\begin{corollary}
Let $\mathcal{L}$ be a subspace lattice on a Banach space $X$ and
$\mathcal{A}=alg\mathcal{L}$. If there exists a non-trivial
idempotent $P\in \mathcal{A}$ such that $ran(P)\in \mathcal{L}$ and
$PB(X)(I-P)\subseteq \mathcal{A}$, then every $C \in \mathcal{A}$ is
a Jordan higher all-derivable point.
\end{corollary}

Combining Theorem \ref{0006} with \cite[Theorem 3.3]{PAN}, we have

\begin{corollary}
Suppose $\mathcal{A}$, $\mathcal{B}$ are as in Corollary \ref{0002} with $\mathcal{B}=\mathcal{A}$. Then every $0\neq C\in \mathcal{A}$ is a higher all-derivable point .
\end{corollary}
We say that $W$ in an algebra $\mathcal{A}$ is a left (or right) separating point of $\mathcal{A}$ if $WA=0$ (or $AW=0$) for $A \in \mathcal{A}$ implies $A=0.$ In \cite[Remark 1]{LI2}, the authors point out that if every Jordan derivation on a unital Banach algebra $\mathcal{A}$ is a derivation, then every linear mapping on $\mathcal{A}$ which is derivable at an arbitrary left or right separating point of $\mathcal{A}$ is a derivation. Together with Theorem \ref{0006}, we may generalize this result to the higher derivation case.

\begin{theorem}
Let $\mathcal{A}$ be a unital Banach algebra such that every
Jordan derivation on $\mathcal{A}$ is a derivation. Suppose that $W$
in $\mathcal{A}$ is a left or right separating point. If
$D=(d_i)_{i\in\mathbb{N}}$ is a family of linear mappings higher
derivable at $W$, then $D=(d_i)_{i\in\mathbb{N}}$ is a higher
derivation.
\end{theorem}

For any non-zero vectors $x\in X$ and $f\in X^*$, the rank one
operator $x\otimes f$ is defined by $x\otimes f(y)=f(y)x$ for $y\in
X$.

\begin{lemma}\label{3002}
If $\mathcal{A}$ is a norm-closed subalgebra of $B(X)$
such that  $\vee \{x: x\otimes f\in \mathcal A\}=X$ and
$\wedge \{\emph{ker}(f): x\otimes f\in \mathcal A\}=(0)$, then every derivation $\delta$ from
$\mathcal{A}$ into $B(X)$ is bounded.
\end{lemma}
\begin{proof}
By the closed graph theorem, it is sufficient to show if $A_n\rightarrow A$ and $\delta(A_n)\rightarrow B$,
as $n\rightarrow \infty$, then $\delta(A)= B$.

For any $x\otimes f, \ y\otimes g \in \mathcal{A}$,
since
\begin{align*}
\delta(x\otimes f A_n y\otimes g)&=f(A_ny)\delta(x\otimes g)\\
&=x\otimes f\delta(A_n y\otimes g)+\delta(x\otimes f)(A_n y\otimes g)\\
&=x\otimes f(\delta(A_n)y\otimes g+A_n \delta(y\otimes g))+\delta(x\otimes f)(A_n y\otimes g),\\
\end{align*}
we have
\begin{eqnarray}
(x\otimes f)\delta(A_n)(y\otimes g)=f(A_ny)\delta(x\otimes g)-(x\otimes f) A_n \delta(y\otimes g)-\delta(x\otimes f)(A_n y\otimes g).\label{3001}
\end{eqnarray}
Taking limit in (\ref{3001}) yields
\begin{eqnarray*}
(x\otimes f) B (y\otimes g)=(x\otimes f)\delta(A)(y\otimes g).
\end{eqnarray*}
Hence $f(By)=f(\delta(A))$. Thus $\delta(A)= B$.
\end{proof}

By \cite{higher2}, if $\{d_i\}_{i\in\mathbb{N}}$ is a Jordan higher derivation on an
algebra $\mathcal{A}$, then there is a
sequence $\{\delta_i\}_{i\in\mathbb{N}}$ of Jordan derivations on $\mathcal{A}$ such that
\begin{eqnarray*}
d_n=\sum_{i=1}^n\left(\sum_{\sum_{j=1}^i r_j=n}\left(\prod_{j=1} ^ i \frac{1}{r_j+\cdots r_i} \right)\delta_{r_1}\ldots\delta_{r_i} \right),
\end{eqnarray*}
where the inner summation is taken over all positive integers $r_j$
with $\sum_{j=1}^i r_j=n$. This together with Lemma \ref{3002}
leads to the following Theorem.

\begin{theorem}\label{3003}
If $\mathcal{A}$ is a norm-closed subalgebra of $B(X)$
such that  $\vee \{x: x\otimes f\in \mathcal A\}=X$ and
$\wedge \{\emph{ker}(f): x\otimes f\in \mathcal A\}=(0)$, then every Jordan
higher derivation on $alg\mathcal{L}$ is bounded.
\end{theorem}
\begin{proof}
Since every Jordan derivation on $\mathcal{A}$ is a derivation by \cite[Theorem 4.1]{LI7}.
\end{proof}

For a subspace lattice $\mathcal{L}$ of a Banach space $X$ and for $E\in \mathcal{L}$, define
$$E_-=\vee\{F\in \mathcal{L}: F\nsupseteq E\}.$$
Put $$\mathcal{J}(\mathcal{L})=\{K\in \mathcal{L}: K\neq(0)~\mathrm{and}~K_-\neq X\}.$$

\begin{remark}
\emph{ It is well known (see \cite{ORO}) that $x\otimes f\in
\textrm{alg}\mathcal{L}$ if and only if there exists some $K\in
\mathcal{J}(\mathcal{L})$ such that $x\in K$ and $f\in K_-^\perp$.
It follows that if a subspace lattice $ \mathcal{L}$ satisfies
$\vee\{K: K\in \mathcal{J}(\mathcal{L})\}=H$ and
$\wedge\{K_-: K\in \mathcal{J}(\mathcal{L})\}=(0)$, then  $alg\mathcal{L}$
satisfies the hypothesis of  Theorem \ref{3003}. Such subspace lattices include
completely distributive subspace lattices, $\mathcal{J}$-subspace lattices, and
subspace lattices with $H_-\neq H$ and $(0)_+\neq (0)$.
Recall that (see \cite{SRL}), a subspace lattice $\mathcal{L}$ is called
\textit{completely~distributive} if $L=\vee\{E\in \mathcal{L}:
E_-\ngeq L\}$ and $L=\wedge\{E_-: E\in
\mathcal{L}~\mathrm{and}~E\nleq L\}$  for all $L\in \mathcal{L}$. It
follows that completely distributive subspace lattices satisfy the
conditions $\vee\{K: K\in \mathcal{J}(\mathcal{L})\}=H$ and
$\wedge\{K_-: K\in \mathcal{J}(\mathcal{L})\}=(0)$. A subspace lattice $\mathcal{L}$ is called a
\textit{$\mathcal{J}$-subspace lattice} on $H$ if $\vee\{K: K\in
\mathcal{J}(\mathcal{L})\}=H$, $\wedge\{K_-: K\in
\mathcal{J}(\mathcal{L})\}=(0)$, $K\vee K_-=H$ and $K\wedge
K_-=(0)~\mathrm{for~any}~K\in \mathcal{J}(\mathcal{L})$.}
\end{remark}

On October 22, 2011, in the Third Operator Theory and Operator
Algebras Conference of China, the first author reported main results
of the paper.

\section*{Acknowledgement}
This work is supported by NSF of China and the Ruth and Ted Braun
Fellowship from the Saginaw Community Foundation.

\end{document}